\documentclass{amsart}
\usepackage{mathtools}
\usepackage{amsfonts}
\usepackage{amssymb}

\usepackage[shortlabels]{enumitem}
\setlist{leftmargin=*}

\usepackage[sorted,abbrev]{amsrefs}

\newtheorem{theorem}{Theorem}[section]
\newtheorem{corollary}[theorem]{Corollary}
\newtheorem{lemma}[theorem]{Lemma}

\newtheorem{fact}[theorem]{Fact}
\newtheorem{ask}{Question}

\theoremstyle{definition}
\newtheorem{definition}{Definition}[section]

\theoremstyle{remark}

\newtheorem{remark}[theorem]{Remark}
\newtheorem{example}[theorem]{Example}

\DeclareMathOperator{\GL}{GL}
\gdef\gl{\mathfrak{gl}}
\gdef\fg{\mathfrak{g}}

\gdef\GG{\mathbb{G}}
\gdef\TT{\mathbb{T}}
\gdef\RR{\mathbb{R}}
\gdef\DD{\mathbb{D}}
\gdef\HH{\mathbb{H}}
\gdef\SS{\mathbb{S}}

\gdef\P{\mathbb P}

\title{Solvable Lie groups definable in o-minimal theories}

\author{Annalisa Conversano}
\address{Annalisa Conversano, Massey University Albany, INMS, IIMS Building, Private Bag 102904, North Shore City 0745, New Zealand}
\email{a.conversano@massey.ac.nz}

\author{Alf Onshuus}
\address{ Alf Onshuus\\ Departamento de Matem\'aticas \\Universidad de los Andes \\
   Cra 1 No. 18A-10, Edificio H\\
   Bogot\'a, 111711\\
   Colombia}
\email{aonshuus@uniandes.edu.co}

\author{Sergei Starchenko}
\address{ Sergei Starchenko, Department of Mathematics, University of
  Notre Dame, Notre Dame, IN 46556}
\email{starchenko.1@nd.edu}

\thanks{The third author was partially supported by NSF}



\begin{document}

\begin{abstract}
In this paper we completely characterize solvable real Lie groups
definable in o-minimal expansions of the real field.
\end{abstract}

\maketitle

\section{Introduction}
\label{sec:introduction}

It is proved in \cite{Pi88} that every group definable in an o-minimal
expansion of the real field is a Lie group (i.e. it can be equipped with
a smooth manifold structure so that the group operations are smooth),
and it is natural to ask when the opposite is true.

\begin{ask}\label{ask:main}
  What real  Lie groups are Lie isomorphic to groups definable in o-minimal
  expansions of the real field?
\end{ask}

A significance of the above question is that for groups definable in
o-minimal  structures one can use the well developed and powerfull tools of
o-minimality (see \cite{LVD} and \cite{DM} for more details on o-minimality).

\medskip

In this paper we give a complete answer to  Question \ref{ask:main} for solvable
Lie groups.  Let $G$ be a solvable connected Lie group. Assume $G$
is  definable in an o-minimal expansion of the real
field. Then, by properties of o-minimal groups (see Fact \ref{CP})
$G$ contains a definable normal torsion-free subgroup $H$ such that
$G/H$ is compact.

First we  prove  (see Theorem \ref{completely solvableThm}) that every torsion-free group definable
in an o-minimal structure is completely solvable.  It follows then that  a necessary
condition for a connected solvable Lie group $G$ to be   Lie
isomorphic to a group definable in an o-minimal expansion of the real
field is that $G$   contains  a  torsion-free completely solvable  subgroup $N$ such that $G/N$ is compact;
and we prove (see Theorem \ref{thm:solv-lie-def}) that this condition is sufficient.
Moreover every solvable Lie group  satisfying this condition is Lie
isomorphic to a group definable in the structure ${\mathbb
  R}_{an,exp}$.

\medskip

By the Levi decomposition, any connected Lie group $G$ is the product of its solvable radical $R$ and
a semisimple subgroup $S$. We know by results in \cite{PPS00} and
\cite{PPS} that any semisimple Lie group is Lie  isomorphic to a group
definable in an o-minimal expansion of the field  of reals if and only
if it has a finite center.  Thus to answer Question \ref{ask:main}
in the whole generality  one needs   to understand definability  of  actions of semi-simple groups  on
solvable  groups.

\subsection{The structure of the paper}
\label{sec:structure-paper}

In Section \ref{sec:preliminaries}
we recall basic facts about groups definable in o-minimal structures
and Lie groups that we will need in this paper.

In Section \ref{completely solvable} we prove that every  completely
solvable connected torsion-free real
Lie group is Lie isomorphic to a group definable in the structure
${\mathbb R}_{exp}$ (see Theorem
\ref{LieAreDef}).

In Section \ref{Appendix} we prove that every torsion-free group
definable in an \emph{arbitrary} o-minimal structure is completely
solvable (see Theorem \ref{completely solvableThm}).

In Section \ref{SS:Lie groups} we answer Question \ref{ask:main} for
solvable Lie groups (see Theorem \ref{thm:solv-lie-def} there).

\section{Preliminaries.}
\label{sec:preliminaries}

\subsection{Solvable and completely solvable Lie groups}
\label{sec:supers-lie-groups}

In this paper by a Lie group we always mean a real Lie group.

The following fact follows from Theorem 3.1 and Corollary 1 in  \cite[Chapter 2]{OnishIII}.
\begin{fact} \label{fact:tor-free-simply}
For a connected solvable Lie group $G$ the following are
  equivalent.
  \begin{enumerate}
  \item  $G$ is torsion-free.
  \item $G$ is simply connected.
  \item $G$ is diffeomorphic to $\RR^n$.
  \end{enumerate}
\end{fact}

Thus for connected solvable Lie group $G$ we will use torsion-free and
simply connected interchangeably.

\medskip
We now turn to completely solvable Lie groups and Lie algebras.

\begin{definition} Let $\fg$ be a Lie algebra.
\begin{enumerate}
\item  \emph{A flag of ideals} in $\fg$ is a chain
\[  \fg=\fg_n > \fg_{n-1} > \dotsb >\fg_0 = 0\]
such that each $\fg_i$ is an ideal of $\fg$.
\item A flag of ideals $\fg=\fg_n > \fg_{n-1} > \dotsb >\fg_0 = 0$ is called
  \emph{complete} if $\dim(\fg_i)=i$ for each $i=0,\dotsc,n$.

\item  A real Lie algebra $\fg$ is called \emph{completely solvable} (also
often called \emph{split-solvable}) if  it has a complete flag of ideals.
\end{enumerate}
 \end{definition}

A connected Lie group is called \emph{completely solvable} (also often
called \emph{triangular} or \emph{split-solvable}) if its
corresponding Lie algebra is.

By the
functorial correspondence between  simply connected  Lie groups and
their Lie algebras, and using \cite[Section2, Theorem 3.1]{OnishIII}  one obtains the following
alternative definition of connected torsion-free  completely solvable  Lie groups.

\begin{fact}\label{fact:lie-split}
A connected torsion-free  solvable Lie group $G$ is completely solvable if and only
if  there
exist a  sequence of subgroups
\[ G=G_n  > G_{n-1} > \dotsb G_0 =\{ e\} \]
such that each $G_i$ is normal in $G$ and $G_{i+1}/G_{i}$ is one-dimensional simply connected Lie group for $i<n$.
\end{fact}

The following is a well-known example of a connected torsion-free
solvable group $\widetilde{E}^0(2)$ that is not completely solvable
(see \cite[Chapter 2, Section 6.4]{OnishIII} and also \cite[Chapter 1,
  Section 1, Example 12(c)]{Knapp}).
  \begin{example}\label{sample:noncompl-solv}
Let    $ \widetilde{E}^0(2)$ be the semi-direct product
$ \widetilde{E}^0(2) = \mathbb{R}^2 \rtimes_{\gamma} \mathbb{R}$, where for $x \in \mathbb{R}$
\[
\gamma(x) = \begin{bmatrix} \ \  \cos 2 \pi x & \sin 2 \pi x \\
-\sin 2 \pi x & \cos 2 \pi x \end{bmatrix}.
\]
The group $\widetilde{E}^0(2)$ is connected torsion-free solvable
group that is not completely solvable. It is a simply connected group
with the Lie algebra of all matrices
\[
\begin{pmatrix}
0 & \theta & x \\
-\theta  & 0 & y \\
0 & 0 & 0
\end{pmatrix}.
\]
\end{example}

\subsection{Groups definable in o-minimal structures}
\label{sec:groups-definable-o}

We refer to \cite{vdd} for basics on o-minimal structures.

In this paper we will need two particular o-minimal expansions of the
real field.

\begin{fact}[See \cite{wilkie}]
The expansion of the real field by
the exponential function $e^x$ is o-minimal. (This structure is denoted
by ${\mathbb R}_{exp}$.)
\end{fact}

\begin{fact}[See  \cite{MMV}]
The expansion of the real field by the exponential function and all restrictions of
analytic functions to compact domains is o-minimal.  (This structure
is denoted by ${\mathbb R}_{an, exp}$.)
\end{fact}

\medskip

We now turn to groups definable in o-minimal structures.

The following two facts are  proved in \cite{Pi88}.
\begin{fact}\label{LieGroups}
Let $G$ be a group definable in an o-minimal structure. Then $G$ can
be equipped with a definable topological manifold structure so that the group
operations  are continuous. In particular,
if $G$ is definable in an o-minimal expansion of the real field
then $G$ is a Lie group.
\end{fact}

Using the above fact we will always view  groups definable in
o-minimal structures as definable topological groups.

\begin{fact}\label{fact:defsubgr}
Let $G$ be a group definable in an o-minimal structure.
If $H<G$ is a definable subgroup then $H$ is closed in $G$.
In particular,
if $G$ is definable in an o-minimal expansion of the real field and
$H<G$ is a definable subgroup then $H$ is  a Lie subgroup of $G$.
\end{fact}

The following is Theorem 1.2 in \cite{PeSte}.

\begin{fact}\label{existence of o-minimal}
If  $G$ is a torsion-free group definable in an o-minimal structure
then $G$ contais a definable one-dimensional  subgroup.
\end{fact}

The following fact follows from  Corollary 2.15 in \cite{Pi88}.
\begin{fact}\label{fact:one-dim}
 If $G$ is a definably connected one-dimensional group definable in an
 o-minimal structure then $G$ is abelian.
\end{fact}

Using the following  ``definable choice'' for groups,  proved  by M.~
Edmundo (see \cite[Theorem 7.2]{Ed}), we will always view quotients of
groups definable in an o-minimal structures as definable objects.

\begin{fact}\label{PillayImaginaries}
Let $G$ be a  group definable in an o-minimal structure and let $\{T(x)\mid x\in X\}$
be a definable family of non empty definable subsets of $G$. Then there is a
definable function $t:X\rightarrow G$ such that for all $x,y\in X$,
we have $t(x)\in T(x)$ and if $T(x)=T(y)$ then $t(x)=t(y)$.
\end{fact}

In the following fact we collect properties of definable torsion-free
groups that we will need in this paper.  For proofs we refer to
\cite{Ed}  and also to \cite{PeSt05}.

\begin{fact}\label{fact:tor-free}
  Let $\mathcal M$ be an o-minimal structure and $G$ a torsion-free group
  definable in $\mathcal M$.
  \begin{enumerate}
  \item $G$ is definably connected and solvable. More precisely,  there are \emph{definable}
    subgroups
 \[ G=G_n> G_{n-1} > \dotsb > G_0=\{ e\} \]
such that $G_i$ is normal in $ G_{i+1}$ and the group $G_{i+1}/G_i$ is
a \emph{torsion-free} abelian group for $i=0,\dotsc,n-1$.
\item  There is a definable normal subgroup $H< G$ such that  the
  group $G/H$ is one-dimensional (hence abelian).
\item If $H<G$ is a definable normal subgroup then the group $G/H$ is
  also torsion-free.
\item If $\mathcal M$ is an expansion of a real closed field then $G$
  is definably diffeomorphic to $M^n$.
  \end{enumerate}
\end{fact}

By analogy with simply connected Lie groups we define definably completely solvable groups.

\begin{definition}\label{completely solvableDef}
A torsion-free group $G$ definable in an o-minimal structure  is
called  \emph{definably completely solvable} if there exist a sequence of definable subgroups
$G=G_n > G_{n-1} > \dots  > G_0=\{e\}$ such that each $G_i$ is normal
in $G$ and $G_{i+1}/G_{i}$ is a one-dimensional  group.
\end{definition}
\begin{remark} It follows from Facts \ref{fact:one-dim} and
  \ref{fact:tor-free} that in the  above definition all groups
  $G_{i+1}/G_{i}$ are abelian and torsion-free.

\end{remark}

\section{Solvable connected torsion-free  Lie groups are definable in ${\mathbb R}_{exp}$}
\label{completely solvable}

In this section, we will prove that
any connected torsion-free   solvable Lie group of finite dimension is isomorphic (as a Lie group) to a
group definable in ${\mathbb R}_{exp}$.

In order to prove this we will need the following lemma about completely
solvable  Lie algebras
over $\mathbb R$.

\begin{lemma}\label{Dixmier2}
Let $\fg$ be a solvable finite dimensional   Lie algebra over  $\mathbb R$. The following are equivalent.
\begin{enumerate}
\item $\fg$ is completely solvable.
\item For any $\xi\in \fg$ all  eigenvalues of the linear operator $ad(\xi)$ are in $\mathbb R$.
\item $\fg$ is isomorphic to a subalgebra of the upper triangular matrices $t_n(\mathbb R)$ for some $n\in \mathbb N$.
\end{enumerate}
\end{lemma}
\begin{proof}

For the  equivalence of $(1)$ and $(2)$ we refer to \cite[Corollary
1.30]{Knapp}.

The implication $(3)\Rightarrow (1)$ is easy.

Thus we only need to see that $(1)$ implies $(3)$. Although this
implication is stated in several books, e.g. in \cite[Section
2]{OnishIII}, we could not find a reference for a proof of it. Here we present an
argument   provided to us by E.B.~Vinberg in a private communication.

\medskip

\noindent  $(1)\Longrightarrow (3)$.
Using Ado's  theorem we can embed $\fg$ into $\gl(V)$ for
some finite dimensional $\mathbb R$-vector space $V$, and we will  assume
that $\fg$ is a Lie subalgebra of $\gl(V)$.

Let $\fg^a$ be the minimal algebraic subalgebra of $\gl(V)$ containing
$\fg$,  i.e. $\fg^a$ is the Lie algebra of an algebraic subgroup of
$\GL(V)$, $\fg^a$ contains $\fg$ and is the minimal such. (It exists by a
dimension argument.)   By
\cite[Chapter 1, Theorem 6.2]{OnishIII} (see also \cite[Theorem 4.1]{PPS}) we
have $[\fg^a,\fg^a] =[\fg,\fg]$, in particular $\fg$ is an ideal of $\fg^a$ and  $\fg^a/\fg$ is abelian.

Let $\GG<\GL(V)$ be a connected  algebraic subgroup whose Lie algebra is $\fg^a$.
Clearly  $\GG$ is defined over $\mathbb R$.

Let $\fg=\fg_n>\fg_{n-1} >\dotsb>\fg_0=0$ be a complete flag of ideals in
$\fg$.  Consider the adjoint action of $\GG$ on $\fg^a$.  Let
\[ G'=\{ g\in \GG \colon Ad(g)(\fg_i) \subseteq \fg_i \text{ for each } i=0,\dotsc,n\}.\]
$G'$ is an algebraic subgroup of $\GG$ and,  by \cite[Capter 2,
Proposition 1.4]{OnishI}, its Lie algebra is
\[ {\mathfrak{g}}'=\{ \xi \in \fg^a \colon [\xi,\fg_i] \subseteq \fg_i \text{ for each } i=0,\dotsc,n\}.\]
Obviously $\fg\subseteq \mathfrak{g}'$, so by minimality of $\fg^a$ we
obtain ${\mathfrak{g}}'=\fg_a$ and $G'=\GG$.

Thus each $\fg_i$ is an ideal of $\fg^a$, and since $\fg/\fg^a$ is abelian,
the chain $\fg=\fg_n>\fg_{n-1} >\dotsb>\fg_0=0$ can be extended to a complete
flag of ideals of $\fg^a$. Hence $\fg^a$ is completely solvable.
Replacing $\fg$ by $\fg^a$ if needed we will  assume that $\fg$ is the Lie
subalgebra of an algebraic group $\GG$ defined over $\RR$.

Let $\GG_u$ be the subset of $\GG$ consisting of unipotent elements, and
$\TT<\GG$ a maximal algebraic torus defined over $\mathbb R$ (as usual by
an algebraic torus we mean a commutative algebraic group consisting of
semi-simple elements).
Then, by \cite[Theorem III.10.6]{Borel},  $\GG_u$ is a normal subgroup
of $\GG$ and $\GG=\TT\cdot \GG_u$ (a semi-direct product).

We now consider $\RR$-points of groups $\GG$, $\GG_u$ and $\TT$.
We have $\GG(\RR)=\TT(\RR)\cdot \GG_u(\RR)$,  $\GG_u(\RR)$ is a normal
subgroup of $\GG(\RR)$, and $\fg$ is the Lie algebra of $\GG(\RR)$ viewed as
a Lie group.

We can write $\TT$ as a product $\TT=\DD\cdot\SS$ of algebraic
groups $\DD$, $\SS$
defined over $\RR$ such that $\DD(\RR)$ is isomorphic over $\RR$ to a
product of multiplicative group $\GG_m(\RR)$ and $\SS(\RR)$ is compact.

We have $\GG(\RR)=\SS(\RR)\cdot \DD(\RR)\cdot \GG_u(\RR)$. Let $\HH=\DD\cdot
\GG_u$.  Then $\HH(\RR)$  is a normal subgroup of $\GG(\RR)$ and
$\GG(\RR)=\SS(\RR)\cdot \HH(\RR)$ with finite intersecton $\SS(\RR)\cap \HH(\RR)$.
It is not hard to see that $\HH$ is
an $\RR$-split solvable group, hence, by \cite[Theorem V.15.4]{Borel},
we can choose a basis in $V$ so that every matrix in $H(\RR)$ is upper
triangular.

Consider the adjoint representation  $Ad\colon \GG(\RR)\to \GL(\fg)$.
The differential of this representation is the adjoint representation $ad\colon \fg\to
\mathrm{End}(\fg)$ of $\fg$. By the equivalence of (1)
and (2), for every $\xi\in \fg$ all  eigenvalues of  $ad(\xi)$ are
real. Hence, by \cite[Corollary
1.30]{Knapp}, there is a basis of $\fg$ such that  $ad(\fg)$ consists of
upper triangular matrices, hence $Ad(\GG(\RR))$ consists of upper triangular
matrices as well.  Since the group of upper triangular matrices does
not have non-trivial compact subgroup it implies that the image of
$\SS(\RR)$ under adjoint representation is trivial, hence $\SS(\RR)$
is in the center of $\GG(\RR)$, in particular it is a normal subgroup
of $\GG(\RR)$.

Let $\mathfrak{s}$ be the Lie algebra of $\SS(\RR)$ and $\mathfrak{h}$
be the Lie algebra of $\HH(\RR)$.  Both $\mathfrak{s}$ and
$\mathfrak{h}$ are ideals of $\fg$ and $\fg$ is the direct sum
$\fg=\mathfrak{s}\oplus\mathfrak{h}$.   We know already that (after
choosing a suitable basis) $H$ consist of upper triangular matrices,
hence $\mathfrak{h}$ also consists of upper triangular matrices, so
$\mathfrak{h}$ has a faithful representation by upper triangular
matrices. The Lie algebra $\mathfrak{s}$ is abelian, hence it also has
a faithful representation by upper-triangular matrices, e.g. by
diagonal matrices. Taking the direct sum of these two representations we
obtain a faithful representation of  $\fg$ by upper triangular
matrices.
\end{proof}

We will also need the following fact due to Dixmier (\cite{Di57}).

\begin{fact}\label{Dixmier1} Let $G$ be a connected Lie group
  whose Lie algebra $\frak g$ is completely solvable. Then the
  exponential map $\exp_{\frak g}\colon \frak g \to G$
is surjective. If in addition $G$ is simply connected then $\exp_{\frak
  g}$ is a diffeomorphism.
 \end{fact}

 \begin{remark}\label{rem:exp-upper-tr}
 For $n\in \mathbb N$ let $T^+_n(\mathbb R)$ be the
   group of all upper triangular $n\times n$-matrices with positive
   diagonal entries.  Its Lie algebra is the algebra of all upper
   triangular $t_n(\RR)$.  Let $Exp_n\colon \gl(n,\RR)\to \GL(n,\RR)$ be
   the usual matrix exponentiation
\[
Exp_n(M)=\sum_{i=0}^\infty \frac{1}{n!}{M^n}.
\]
It follows from Fact \ref{Dixmier1} (and also can be seen by direct
computations) that $Exp_n$ maps $t_n(\RR)$ \emph{diffeomorphically}
onto $T^+_n(\RR)$. Also  for any Lie subgroup $G < T^+_n(\RR)$ the
matrix exponentiation  $Exp_n$
maps the Lie algebra $\fg$ of $G$ diffeomorphically onto $G$.
 \end{remark}

\begin{lemma}\label{lemma:com-tr}
Let $G$ be a connected torsion-free completely solvable  Lie group. Then $G$
is Lie isomorphic to a Lie subgroup of $T^+_n(\RR)$ for some $n\in
\mathbb N$.
\end{lemma}
\begin{proof}
 Let $\frak g$ be the Lie algebra of $G$. Since $\frak g$ is
 completely solvable, by  Lemma \ref{Dixmier2}, there is an embedding
$\varphi\colon \mathfrak g \to t_n(\RR)$ for some $n\in \mathbb N$.

By the general theory of Lie groups and Lie algebras, there is a smooth
group homomorphsim $\Phi\colon G\to T^+_n(\RR)$ whose differential is
$\varphi$.

By \cite[Section 2, Theorem 2.4]{OnishIII}, the image $\Phi(G)$ is a
simply connected Lie subgroup of $T^+_n(\RR)$.
Since a lift of an isomorphism between Lie
algebras to corresponding simply connected Lie groups is an
isomorphism,  $\Phi$ is a Lie  isomorphism between $G$ and $\Phi(G)$.
\end{proof}

 \begin{lemma}\label{lem:exp-def-in-m}
For every $n\in \mathbb N$ the restriction of the map $Exp_n$ to $t_n(\RR)$ is definable in $\RR_{exp}$.
 \end{lemma}
 \begin{proof} Using Jordan normal form, because the diagonal and idempotent components commute, it
   is easy to that the restriction of $Exp_n$ to
   the set of  $n\times n$-matrices whose all eigenvalues are real
   is definable in $\RR_{exp}$.
  \end{proof}

\begin{theorem}\label{LieAreDef}
Any connected torsion-free  solvable  Lie group is Lie isomorphic
to a group definable in $\mathbb R_{exp}$.
\end{theorem}
\begin{proof}
Let $G$ be a connected torsion-free  solvable   Lie group.
 By Lemma \ref{lemma:com-tr}, $G$ is Lie isomorphic to some Lie group
 $G'\subseteq T_n^+(\RR)$.  Let $\mathfrak{g'}\subseteq t_n(\RR)$ be
   the Lie algebra of $G'$.  Then $G'=Exp_n(\mathfrak g')$ by Fact
   \ref{Dixmier1}, and $G'$  is definable in $\RR_{exp}$ by Lemma \ref{lem:exp-def-in-m}.
 \end{proof}

\section{Torsion-free  solvable o-minimal groups are completely solvable}\label{Appendix}

In this section we will prove that every torsion-free  group
definable  in \emph{any} o-minimal structure is definably completely solvable.

\bigskip

We will need the following result by Baro, Jaligot and Otero, which is
Corollary 6.8 in \cite{BaroJagOt} that guarantees existence of infinite abelian definable normal subgroups in any definable torsion-free group.

\begin{fact}\label{BaJaOt}
Let $G$ be a connected solvable group definable in an o-minimal structure. Then $[G, G]$ is definable and
definably connected and, for each $n\geq 1$, the $n$-th derived subgroup $G^{(n)}$ is definable and definably connected.
\end{fact}

From the above  fact we obtain the following corollary.

\begin{corollary}\label{an abelian group}
Let $G$ be a torsion-free group definable in an o-minimal
structure. Then it contains an infinite  abelian normal definable subgroup.
\end{corollary}

\begin{proof}
By Fact \ref{fact:tor-free} $G$ is solvable, and by Fact \ref{BaJaOt} all of the derived subgroups are definable. Since $G$ is finite dimensional, its
derived series is finite so we have  least $n\in \mathbb N$ such that $G^{(n+1)}=0$ for some $n$. But then $G^{(n)}$
is an abelian normal subgroup of $G$, as required.
\end{proof}
\bigskip

We will also need the following fact that follows from
Theorem A and  Theorem 4.1 in \cite{MiSt}.



\begin{fact}\label{Corollary to MiSt}
Let $\mathcal R$ be an o-minimal expansion of a real closed field
with additive group $(R,+)$ and multiplicative group $(R^{>0},
\cdot)$. Then either $(R,+)$ is $\mathcal R$-definably isomorphic to
$(R^{>0}, \cdot)$, or every definable subgroup of  a Cartesian power $(R^{>0},\cdot)^l$ is 0-definable.
\end{fact}

\medskip
The following lemma about definable abelian subgroups of centerless
definable groups will be a key.

\begin{lemma}\label{P(A) for linear}
Let $\mathcal R$ be an o-minimal expansion of a real closed field, let
$G$ be a centerless definably connected definable subgroup of
$\GL_n(\mathcal R)$ for some $n\in \mathbb N$, and let $A< G$ be a
definable torsion-free  normal abelian subgroup of $G$. Then $A$ is
$\mathcal R$-definably isomorphic to a Cartesian power of the additive
group of $\mathcal R$.
\end{lemma}

\begin{proof} Since $A$ is a linear group,
it follows from Proposition 3.8 and Lemma 3.7 in \cite{PPS} that
there are definable subgroups $A_m$ and $A_a$ of $A$ such that
$A=A_a\times A_m$,  the group $A_a$ is definably isomorphic to a Cartesian
power $(R,+)^k$ and $A_m$ is definably isomorphic to a Cartesian power $(R^{>0}, \cdot)^l$.
If $(R, +)$ is definably isomorphic to $(R^{>0}, \cdot)$, there is
nothing to prove.

Assume $(R, +)$ and $(R^{>0}, \cdot)$ are not
definably isomorphic.  It is easy to see then that for any definable
automorphism $\sigma$ of $A$ we have $\sigma(A_a)=A_a$ and
$\sigma(A_m)=A_m$. Considering action of $G$ on $A$ by conjugation we
obtain that both $A_a$ and $A_m$ are normal subgroups of $G$.  By Fact
\ref{Corollary to MiSt} every uniformly definable family of
automorphisms  of $A_m$ is finite.  Since $H$ is connected, it implies
that the action of $H$ on $A_m$ is trivial, hence $A_m$ is in the
center of $H$, but $H$ is centerless. Thus $A_m$ is trivial and $A=A_a$.
\end{proof}

We can now prove the main result of this section.

\begin{theorem}\label{completely solvableThm}
Every torsion-free  group definable in an o-minimal structure  is completely solvable.
\end{theorem}
\begin{proof}
By  Corollary 2.3 in \cite{PeSt05}, if $G$ is a torsion-free group
definable in an o-minimal structure and $H<G$ is a definable normal subgroup,
then the factor group  $G/N$ is
torsion-free as well. Thus,  by an easy induction on dimension,  it is
sufficient to show that every torsion-free group $G$ definable in an
o-minimal structure   has a definable one-dimensional (hence abelian) normal
subgroup.

We will prove an existence of a definable one-dimensional normal
subgroup by induction on the dimension of $G$.

\medskip
If $\dim(G)=1$ then there is nothing to prove.

\medskip
Let $G$ be a torsion-free group definable in o-minimal structure
$\mathcal M$ with $\dim(G)=k$ and we assume that every torsion-free
group definable in $\mathcal M$ of dimension less than $k$ has a
definable normal one-dimensional subgroup. We need to show that $G$
has a definable normal one-dimensional subgroup.

If $G$ is abelian then we are done by Fact \ref{existence of
  o-minimal}.  Assume $G$ is not abelian.  If the center $Z(G)$ is  non-trivial, then 
$\dim(Z(G))<k$, hence $Z(G)$ contains a definable normal
one-dimensional subgroup $A$.  Obviously $A$ is also normal in $G$.

Thus we may assume that $G$ is centerless.
If $G$ is a direct product $G=G_1\otimes G_2$ of definable proper
subgroups then, by induction hypothesis, $G_1$ has a definable normal
one-dimensional subgroup, and this subgroup is normal in $G$.  Thus we
may assume that $G$ is not a direct product of definable proper subgroups.
It follows then from  Theorems 3.1 and 3.2 in \cite{PPS00}
that there is a real closed field $R$ definable in $\mathcal M$ such that $G$
is definably isomorphic to a subgroup of $\GL(n,R)$.  Hence we may
assume that $\mathcal M$ is an expansion of a real closed field $R$
and $G$ is a definable subgroup of $\GL(n,R)$.

By Corollary \ref{an abelian group}, $G$ contains a nontrivial normal
definable abelian subgroup $A$, and by Lemma \ref{P(A) for linear} the
group $A$ is definably isomorphic to a Cartesian power $(R,+)^l$.

Let $\P(A)$ be the set of all definable one-dimensional subgroups of
$A$. Since $A$ is definably isomorphic to $(R,+)^l$ the set $\P(A)$
can be identified with the projective space $\P^l(R)$. In particular
it is definable and definably compact.

The group $G$ acts on $A$ by conjugation and this action induces a
continuous action of $G$ on the set $\P(A)$.  To finish the proof of the theorem it is
sufficient to show that under this action $G$ has a fixed point in
$\P(A)$.  It will follow from the following general lemma.

\begin{lemma}\label{sergei2}
Let $H$ be a torsion-free group definable
  in an o-minimal expansion of a real closed field.  Assume $H$ acts definably and
  continuously on a definably compact set $X$. Then $H$ has a fixed
  point in $X$.
\end{lemma}
\begin{proof}
We will do  induction on the dimension of $H$.  For $h\in H$ and
$x\in X$ we will denote by $h.x$ the image of $x$ under the action of
$h$.

\medskip
Assume $\dim(H)=1$.  Then, by \cite{str}, $H$ is definably
homeomorphic to the interval $(0,1)$. Let $x$ be an arbitrary element of $X$. Since $X$ is definably
compact, the limit $\lim_{t\to 1} t.x$ exists in $X$, and it is not
hard to see that this limit is fixed by $H$.

\medskip
Assume  $\dim(H)>1$.
By Fact \ref{fact:tor-free}(2), $H$ has a
definable normal subgroup $K<H$ with $\dim(H/K)=1$.  By induction
hypothesis, $K$ has a fixed point in $X$.  Let $X'\subseteq X$ be
the subset of all points in $X$ fixed by $K$.  By the continuity
of the action, it is a closed subset of $X$, and since $K$ is a normal
subgroup of $H$ the set $X'$  is $H$-invariant.

The action of $H$ on $X'$ induces an action of
$H/K$ on $X'$. Since $H/K$ is one-dimensional,  it has a fixed point in
$X'$, and this point is fixed by $H$.
\end{proof}

This finishes the proof of Theorem \ref{completely solvableThm}.
\end{proof}

Combining Theorems \ref{LieAreDef} and \ref{completely solvableThm} we
obtain a complete description of torsion-free solvable groups
definable in o-minimal expansions of the real field.

\begin{theorem}
  \label{thm:tor-free-mail}
For a connected torsion-free solvable Lie group $G$ the following are
equivalent.
\begin{enumerate}
\item $G$ is Lie isomorphic to a group definable in $\RR_{exp}$.
\item $G$ is Lie isomorphic to a group definable in an o-minimal
  expansion of the real field.
\item $G$ is completely solvable.
\end{enumerate}
\end{theorem}

\section{Extensions of compact groups by torsion-free groups}\label{SS:Lie groups}

Theorem \ref{thm:tor-free-mail} provides a complete characterization
of connected torsion-free solvable Lie groups  definable in  o-minimal
expansions of the real field. In this section we extend it to a
characterization of solvable Lie groups.

The following fact follows from  \cite[Propositions 2.1 and 2.2]{CP}.

\begin{fact} \label{CP}
Let $G$ be a grop definable in an o-minimal structure. Then $G$  contains a maximal normal definable torsion-free subgroup $H$.

In addition, if $G$ is solvable and $H<G$ is the maximal normal definable
torsion-free subgroup then the group $G/H$ is definably compact.
\end{fact}

Combining the above fact with Theorem \ref{completely solvableThm} and
using Fact \ref{fact:defsubgr} we obtain the following corollary.

\begin{corollary}\label{cor:solv-Lie-def} Let $G$ be a solvable Lie group Lie isomorphic to a group
definable in an o-minimal expansion of the real field. Then $G$
contains a normal Lie subgroup $H$ such that
\begin{enumerate}[(a)]
\item The group $H$ is connected torsion-free and completely solvable.
\item The factor group $G/H$ is compact.
\end{enumerate}
\end{corollary}

Our goal is to show that the converse in the above corollary is also
true.

We will need a lemma.

\begin{lemma}\label{lem:lemma2} Let $G$ be a Lie group and $\mathcal
  M$ an o-minimal expansion of $\RR_{an}$. Assume $G$ is a semi-direct
  product of a normal Lie subgroup $H$ and a compact subgroup $K$. If
  $H$ is Lie isomorphic to a group definable in $\mathcal M$ then $G$ also
  is isomorphic to a group definable in $\mathcal M$.
\end{lemma}
\begin{proof} Let $\gamma\colon K\to Aut(H)$ be the group homomorphism
given by the action of $K$ on $H$ by conjugations.  So $G=H \rtimes_{\gamma} K$.

Any compact  Lie group  admits a structure of an algebraic group
(see \cite[Chapter 4, Corollary to Theorem 2.3]{OnishIII}), hence $K$
is Lie isomorphic to a semialgebraic group $K'$ definable in the real
field.
Let $H'$ be a group definable in $\mathcal M$ Lie isomorphic to $H$.
We have that $G$ is Lie isomorphic to $H' \rtimes_{\gamma'} K'$ for
some $\gamma'\colon K'\to Aut(H')$. The group $Aut(H')$ is a Lie group
and the group homorphism $\gamma'$ from the compact Lie group $K'$
into the Lie group $Aut(H')$ is a real analytic map on a compact set,
so it is definable  in $\RR_{an}$.
The corresponding  group isomorphism between $H\rtimes_{\gamma} K$ and
$H'\rtimes_{\gamma'} K'$ is
 a Lie isomorphism as required.
\end{proof}

 \newpage

\begin{theorem}\label{thm:solv-lie-def}
For a solvable Lie group $G$ the following are equivalent.
\begin{enumerate}
\item  $G$ is Lie isomorphic to a group definable in $\RR_{an,exp}$.
\item  $G$ is Lie isomorphic to a group definable in an o-minimal
  expansion of the real field.
\item $G$ contains a normal connected torsion-free completely solvable
  Lie subgroup $H$ such that $G/H$ is compact.
\end{enumerate}
\end{theorem}
\begin{proof}
Implication  $(1)\Rightarrow   (2)$ is
obvious, and  $(2) \Rightarrow   (3)$ is Corollary
\ref{cor:solv-Lie-def}. It
remains to
show that (3) implies (1).

Let $G$ be a solvable Lie group with a normal connected torsion-free
subgroup $H$ such that the group $K=G/H$ is compact.

Since $K$ is a  solvable compact connected  Lie group it is abelian by
\cite[Lemma 2.2]{iwasawa}, hence $G/H$ is abelian and $H$ contains the
commutator subgroup $G'$ of $G$.
Since $H$ is
simply connected, by \cite[Theorem 5.1]{OnishI} and  \cite[Chapter 2, Theorem
3.4(1)]{OnishIII}, $G'$ is closed in $H$, connected  and simply
connected. By a theorem of Malcev (see \cite[Chapter 2, Theorem
7.1]{OnishIII},   $G$ can be decomposed into a semi-direct product
$T\ltimes F$ of a torus $T$ and a simply-connected Lie subgroup $F$.
It is not hard to see that we must have $H=F$.

By Theorem \ref{LieAreDef} the group $H$ is Lie isomorphic to a group definable
in $\RR_{exp}$, and by Lemma \ref{lem:lemma2} $G$ is Lie
isomorphic to a group definable in $\mathcal M = \RR_{an,ex}$.
\end{proof}

\begin{remark}  Let $G$ be a connected real Lie group with compact
  Levi subgroups.  It follows from Lemma \ref{lem:lemma2} and Theorem \ref{thm:solv-lie-def} that $G$ is Lie
  isomorphic to a group definable  in an o-minimal expansion of the
  real field  if and only if its solvable radical $R$ is Lie
  isomorphic to a group definable in $\RR_{an,exp}$.
 \end{remark}


\end{document}